\documentclass[231]{mfat}
\pagespan{37}{50}

\usepackage{mmap}
\usepackage[utf8]{inputenc}

\usepackage[mathcal]{euscript}
\usepackage{mathrsfs}

\newtheorem{theorem}{Theorem}
\newtheorem{lemma}{Lemma}
\newtheorem{proposition}{Proposition}
\newtheorem{corollary}{Corollary}

\theoremstyle{definition}
\newtheorem{definition}{Definition}

\theoremstyle{remark}
\newtheorem{remark}{Remark}

\def\R{\mathbf{R}}
\def\C{\mathbf{C}}
\def\N{\mathbf{N}}
\def\f{\mathscr{F}}
\def\fc{\mathscr{C}}
\def\fcc{\mathscr{K}}
\def\fpd{\mathcal{P}}
\def\D{\mathscr{D}}
\def\FSs{\mathscr{B}}
\def\Fs{\mathscr{A}}
\def\H{\mathscr{H}}
\def\Nn{\mathscr{N}}
\def\supp{\mathrm{supp}\,}
\def\spec{\mathrm{Spec}\,}
\def\id{\mathrm{id}}
\def\meas{\mathscr{M}}
\def\comps{\mathcal{K}}
\def\ucomps#1#2{\mathcal{U}_{#1}(#2)}
\def\var{\mathrm{Var\,}}
\def\act{.\,}
\def\pair#1#2{\langle\, #1,  #2\,\rangle}
\def\bpair#1#2{\Bigl\langle\, #1,  #2\,\Bigr\rangle}
\def\inner#1#2{\bigl(\,#1\,\bigm|\,#2\,\bigr)}
\def\Mat{M}

\begin{document}

\title[On Fourier algebra of a hypergroup $\ldots $]
{On Fourier algebra of a hypergroup constructed from a conditional
expectation on a locally compact group}

\author{A. A. Kalyuzhnyi}
\address{Institute of Mathematics, National Academy of Sciences of Ukraine, 3 Tereshchenkivs'ka,
Kyiv, 01601, Ukraine}
\email{akalyuz@gmail.com}

\author{G. B. Podkolzin}
\address{National Technical University of Ukraine "Kyiv Polytechnic Institute",
37 Prospect Pe\-re\-mogy, Kyiv, 03056, Ukraine}
\email{glebpodkolzin@gmail.com}

\author{Yu. A. Chapovsky}
\address{National Technical University of Ukraine "Kyiv Polytechnic Institute",
37 Prospect Pe\-re\-mogy, Kyiv, 03056, Ukraine}
\email{yury\_chapovsky@mail.ru}

\subjclass[2000]{43A62}
\date{13/11/2016; Revised 28/12/2016}
\keywords{Hypergroup, Fourier algebra, conditional expextation,
comultiplication.}

\begin{abstract}
 We prove that the Fourier space of a hypergroup constructed from a
  conditional expectation on a locally compact group has a Banach
  algebra structure.
\end{abstract}

\maketitle

\section{Introduction}
\label{sec:Introduction}

It is well known~\cite{eymard64:_lalgeb_fourier} that, for a locally
compact group $G$, the Fourier-Stieltjes space $\FSs(G)$ and the
Fourier space $\Fs(G)$, make commutative Banach algebras.

For a hypercomplex
system~\cite{berezansky_kal92:_harmon_analy_hyper_system} or a
DJS-hypergroup~\cite{bloom_heyer91:_harmon_analy}, one can also define
the Fourier-Stieltjes and the Fourier spaces. It is well known that,
being Banach spaces, they do not, in general, make Banach algebras.
Conditions for this to be the case were considered
in~\cite{muruganandam07:_fourier_algeb,muruganandam08:_fourier_algeb}
for commutative hypergroups and special type hypergroups called
ultraspherical. In~\cite{ChKP2015}, a general sufficient condition for
the Fourier and the Fourier-Stieltjes spaces of a locally compact
hypergroup to be Banach algebras was obtained. It was also proved
there that there exists the Fourier Banach algebra on the double coset
hypergroup constructed from a locally compact group.

In this paper we prove that the Fourier space of a DJS-hypergroup that
is constructed from a locally compact group and an expectation, see
\cite{KCondEx, ChKP2011}, also
has a Banach algebra structure. This result is a generalization of the
one obtained in~\cite{muruganandam08:_fourier_algeb}, where the author
imposes additional conditions on the hypergroup.

The paper is organized as follows. In Section~\ref{prelim}, we recall
a traditional definition of a DJS-hypergroup $Q$ in terms of the
convolution algebra structure on bounded measures on $Q$. Then,
in~Propositions~\ref{P:hypergroup->C^*-algebra}
and~\ref{P:C^*-algebra->hypergroup}, we give an equivalent form for a
definition of a DJS-hypergroup in a dual setting in terms of a
comultiplication on the $C^*$-algebra of continuous functions on $Q$
vanishing at infinity and its multiplier algebra of continuous bounded
functions on $Q$. This turns out to be more convenient for purposes of
this paper.

Section~\ref{sec:main-results} gives a main result of the paper. In
Theorem~\ref{Th:hypergroup_from_expectation}, we describe construction
of a DJS-hypergroup from a conditional expectation on a locally
compact group. It follows from proofs that the construction of a
conditional expectation is similar to the construction of orbital
morphisms~\cite{Jewett}, except that the corresponding orbital mapping
need not be open.  Using results from~\cite{ChKP2015} we prove in
Theorem~\ref{T:complete-positivity} the main result that the Fourier
space of the hypergroup constructed from a conditional expectation on
a locally compact hypergroup is a Banach algebra.

\section{Preliminary}
\label{prelim}

\subsection{Hypergroups}\label{hypergroups}

Let $Q$ be a Hausdorff locally compact topo\-logical space.

\smallskip

The linear space of complex-valued continuous functions
(resp., bounded continuous functions, continuous functions approaching
zero at infinity) is denoted by $\fc(Q)$ (resp., $\fc_b(Q)$,
$\fc_0(Q)$).  The space $\fc_b(Q)$ is endowed with the norm
$\|f\|_\infty=\sup_{t\in Q} |f(t)|$.  Support of a function $f\in
\fc(Q)$ is denoted by $\supp(f)$. For an open subset $U$ of $Q$,
$\fcc(U)$ denotes the subspace of functions $f\in \fc(Q)$ such
that
$\supp(f) \subset U$, and $\fcc_+(U)$ (resp., $\fcc_+^*(U)$) denotes
the subset of functions $f\in \fcc(U)$ such that
$f\geq 0$ (resp., $f>0$). For
$p\in U$, we denote by $\fcc_+^p(U)$ the subset of $\fcc_+(U)$ such
that $f(p)>0$. By $1_Q$, we denote the constant function, $1_Q(s)=1$
for all $s \in Q$.

\smallskip A measure is understood as a complex Radon
measure~\cite{bourbaki04:_integ_I} on $Q$. The linear space of complex
Radon measures, over the field $\Bbb C$ of complex numbers, is denoted by
$\meas(Q)$.  The subspace of $\meas(Q)$ of bounded (resp., compactly
supported) measures is denoted by $\meas_b(Q)$ (resp., $\meas_c(Q)$)
and, for $\mu\in \meas_b(Q)$, its norm is $\|\mu\| = \sup_{f\in
  \fcc(Q),\ \|f\|_\infty\leq 1} |\mu(f)|$. The subset of $\meas(Q)$ of
nonnegative (resp., probability) measures is denoted by $\meas_+(Q)$
(resp., $\meas_p(Q)$). For a measure $\mu \in \meas_+(Q)$, its support
is denoted by $\supp(\mu)$. If $\mu \in \meas_+(Q)\cap \meas_b(Q) $,
then $\|\mu\|=\mu(1_Q)$. The Dirac measure at a point $s\in Q$ is
denoted by $\varepsilon_s$. The integral of $f\in \fcc(F)$ with
respect to a measure $\mu \in \meas$ is denoted by
$\mu(f)=\pair{f}{\mu}=\int_F \pair{f}{\varepsilon_t}\, d\mu(t)=\int_F
f(t)\, d\mu(t)$.

\smallskip The set of all nonempty compact subsets of $Q$ is denoted
by $\comps$. We consider $\comps$ endowed with the Michael
topology~\cite{michael51:_topol}. Recall that this topology is
generated by the subbasis of all $\ucomps{V}{W}=\{F\in \comps:
F\subset W \ \hbox{and}\ F\cap V\neq \emptyset\}$ for open subsets $V$
and $W$ of $Q$.

\medskip

A (locally compact)  hypergroup is a locally compact Hausdorff
topological space $Q$ such that $\meas_b(Q)$ is endowed with a
multiplication, called convolution and denoted by~$*$, satisfying the
following conditions~\cite{bloom_heyer91:_harmon_analy}:
\begin{itemize}
\item [$(H_1)$] $(\meas_b(Q), *)$ is an associative algebra over $\Bbb C$.
\item [$(H_2)$] For all $s, t \in Q$, $\varepsilon_s * \varepsilon_t
  \in \meas_p(Q)$ and $\supp(\varepsilon_s*\varepsilon_t) $ is compact.
\item [$(H_3)$] The mapping $(s,t) \mapsto \varepsilon_s *
  \varepsilon_t$ of $Q\times Q$ into $\meas_p(Q)$ is continuous with
  respect to the weak topology $\sigma(\meas_p(Q), \fc_0(Q))$ on $\meas_p(Q)$.
\item [$(H_4)$] The mapping $(s,t)\mapsto
  \supp(\varepsilon_s*\varepsilon_t) $ of $Q\times Q$ into $\comps$ is
  continuous with respect to the Michael topology on $\comps$.
\item [$(H_5)$] There exists a (necessarily unique) element $e \in Q$
  such that $\varepsilon_e*\varepsilon_s=
  \varepsilon_s*\varepsilon_e=\varepsilon_s$ for all $s \in Q$.
\item [$(H_6)$] There exists a (necessarily unique) homeomorphism
  $s\mapsto \check s$ of $Q$ into $Q$ such that $\check{\check s}=s$
  and $(\varepsilon_s*\varepsilon_t)\,\check{}=\varepsilon_{\check t}*
  \varepsilon_{\check s}$, where $\check \mu$ denotes the image of the
  measure $\mu$ with respect to the homeomorphism $s\mapsto \check s$,
  i.e., $\pair{f}{\check \mu} = \pair{\check f}{\mu}$, where $\check f
  (s) = f(\check s)$.
\item [$(H_7)$] For $s,t \in Q$, $e \in
  \supp(\varepsilon_s*\varepsilon_t)$ if and only if $s=\check t$.
\end{itemize}

For a measure $\mu \in \meas(Q)$ and $h\in \fc(Q)$, the measure $h\mu$
is defined by $\pair{f}{h\mu}=\pair{fh}{\mu}$ for $f\in\fcc(Q)$; it is
clear that $h\mu \in \meas_c(Q)$ for $h \in \fcc(Q)$.

Everywhere in the sequel, we assume that the hypergroup possesses a
left invariant measure, denoted by $m$, which means that
\begin{equation*}
  \varepsilon_s*m=m
\end{equation*}
for all $s\in Q$. The Banach space $L_p(Q,m)$, $1\leq p\leq \infty$,
is denoted by $L_p(Q)$.

For $\mu\in \meas_b(Q)$, denote by $\mu^\star$ the bounded measure
defined by $\mu^\star(f)=\overline{\check\mu(\bar f)}$ for $f\in
\fcc(Q)$. It follows from the axiom $(H_6)$ of a hypergroup that
${}^\star$ is an involution on the algebra $(\meas_b(Q), *)$. It is well
known~\cite{bloom_heyer91:_harmon_analy} that $(\meas_b(Q),*,{}^\star)$
is an involutive Banach algebra.

We denote $L_1( m )=\{f m : f \in L_1(Q)\}$. It is well
known that $(L_1( m ),*,{}^\star)$ is a closed two-sided ideal of
$(\meas_b(Q),*,{}^\star)$.

Identifying each $f\in L_1(Q)$ with the measure $f m \in L_1( m )$
yields an involutive Banach algebra structure on $L_1(Q)$, denoted by
$(L_1(Q), *, {}^\star)$, where
\begin{eqnarray*}
  (f*g)(s)&=&\int_Q f(t) \pair{g}{\check \varepsilon_t*\varepsilon_t}\, dm(t),
  \\
  f^\star(s)&=&\kappa^{-1}( s) \overline  f(\check s)
  ,
\end{eqnarray*}
where $\kappa: Q \rightarrow \R$ is the modular function,
$m*\varepsilon_s=\Delta(s) m$, see~\cite{bloom_heyer91:_harmon_analy}. Also denote by
\begin{equation*}
f^\dagger (s) = \overline{f (\check s)} .
\end{equation*}
It is wellknown that a function $f\ast f^\dagger$ is a positive
definite on hypergroup $Q$, see~\cite{bloom_heyer91:_harmon_analy}.

Let $(Q,*,e, \check{}\,)$ be a locally compact hypergroup. Define a
linear map $\tilde \Delta$, homomorphisms $\epsilon$ and $\,\check{}\,$ as
follows:
\begin{equation}\label{eq:Delta-epsilon-tilda-def}
  \begin{array}{c}
    \tilde \Delta\colon \fc_b(Q)\rightarrow \fc_b(Q\times Q),\qquad
    (\tilde \Delta f)(p,q) = \pair{f}{\varepsilon_p \ast
      \varepsilon_q},\\[1.5mm]
    \epsilon\colon \fc_b(Q) \rightarrow \Bbb C,\qquad
    \epsilon(f)=f(e),\\[1.5mm]
    \check{}\,\colon \fc_b(Q) \rightarrow \fc_b(Q),\qquad
    \check f(p)=f(\check p),
  \end{array}
\end{equation}
where $\fc_b(Q)$ is considered as a commutative $C^*$-algebra.

\medskip
The following lemma is clear.
\begin{lemma}\label{L:measure-support-separation}
  Let $Q$ be a locally compact space, $\mu \in \meas_b(Q)\cap
  \meas_+(Q)$, and $F\subset Q$ is a closed set. Then $F\cap
  \supp(\mu)=\emptyset$ if and only if there is $f\in \fc_+(Q)$,
  $\|f\|_\infty=1$, such that $f(r)=0$ for all $r\in F$ and
  $\pair{f}{\mu}=\|\mu\|$.
\end{lemma}

\begin{proposition}\label{P:hypergroup->C^*-algebra}
  Let $(Q,*,e, \check{}\,)$ be a locally compact hypergroup. Then the
  maps defined by~(\ref{eq:Delta-epsilon-tilda-def}) have the
  following properties:
  \begin{itemize}
  \item [$(\widetilde H_1)$] the map $\tilde \Delta$ is linear and
    coassociative, that is,
    \begin{equation*}
      (\tilde \Delta \times \id) \circ \tilde \Delta = (\id \times \tilde \Delta) \circ \tilde \Delta;
    \end{equation*}
  \item [$(\widetilde H_2)$]
    \begin{itemize}
    \item [\textup{(a)}] the map $\tilde \Delta$ is positive, that is,
      $\tilde \Delta (\fc_{b,+})\subset \fc_{b,+}(Q\times Q)$, where
      $\fc_{b,+}(Q)$ \textup{(}resp., $\fc_{b,+}(Q\times Q)$\textup{)}
      is the cone of bounded nonnegative continuous functions on $Q$
      \textup{(}resp., $Q\times Q$\textup{)},
  \item [\textup{(b)}] $\tilde \Delta\,1_Q=1_{Q\times Q}$,
  \item[\textup{(c)}] for any $s,t\in Q$ there is $f\in
    \fcc_+(Q)$ such that $\|f\|_\infty\leq 1$ and $(\tilde \Delta\,
    f)(s,t)=1$;
  \end{itemize}
\item [$(\widetilde H_4)$] if $F\subset Q$ is closed, $(s_0,t_0)\in
  Q\times Q$ and $f_0\in \fcc_+(Q)$ is such that $\|f_0\|_\infty=1$,
  $f_0(r)=0$ for all $r\in F$ and $(\tilde \Delta\, f_0)(s_0,t_0)=1$, then
  there exist an open neighborhood  $U\subset
  Q\times Q$ of $(s_0,t_0)$ and $f\in\fcc_+(Q)$ such that
  $\|f\|_\infty=1$, $f(r)=0$ for all $r\in F$ and $(\tilde \Delta f)(s,t)=1$
  for all $(s,t)\in U$;
  \item [$(\widetilde H_5)$] the homomorphism $\epsilon$ is a right
    and left counit, that is,
    \begin{equation}
      \label{eq:epsilon-counit-property}
      (\epsilon \times \id)\circ \tilde \Delta=(\id \times \epsilon)\circ \tilde \Delta=\id;
    \end{equation}
  \item [$(\widetilde H_6)$] the homomorphism $\, \check\,$ satisfies
    $\fcc(Q)\,\check{}\,\subset \fcc(Q)$ and the following identities:
    \begin{equation}
      \label{eq:involution-property}
      \begin{array}{c}
        \check{}\,\circ\, \check{}\,=\id,\\[1.5mm]
        \tilde \Delta \circ \,\check{}\,=
        \Sigma \circ (\,\check{}\, \times\, \check{}\,)\circ \tilde \Delta,
      \end{array}
    \end{equation}
    where $\Sigma: \fc_b(Q\times Q) \rightarrow \fc_b( Q\times Q)$ is
    defined by $(\Sigma f)(p,q)=f(q,p)$;
  \item [$(\widetilde H_7)$] if $\Gamma\,\check{}\, \subset Q\times Q$
    is a graph of the homeomorphism $\,\check{}\,$, that is,
    \begin{equation*}
      \Gamma\,\check{}\,=\{(p,\check p)\in Q\times Q:p \in Q\},
    \end{equation*}
    then
    \begin{equation}
      \label{eq:graph-involution}
      \Gamma\,\check{}\,=\bigcap_{f\in \fcc_+^e(G)} \supp\bigl(\tilde \Delta(f)\bigr),
    \end{equation}
    where $\fcc_+^e(G)=\{f \in \fcc_+(G): f(e)>0\}$.
  \end{itemize}
\end{proposition}

\begin{proof}
  First of all note that $\tilde \Delta f\in \fc(Q\times Q)$ if $f\in
  \fc(Q)$, as directly follows from ($H_3$). It also immediately
  follows from ($H_2$) that $\tilde \Delta$ is positive and $\tilde \Delta\,
  1_Q=1_{Q\times Q}$. To see that $\tilde \Delta\, f $ is bounded for bounded
  $f\in \fc(Q)$, we have that $-\|f\|_\infty\, 1_Q \leq f \leq
  \|f\|_\infty\, 1_Q$ and, hence,
  \begin{equation*}
    -\|f\|_{\infty}\, 1_{Q\times Q} \leq \tilde \Delta\, f\leq \|f\|_\infty\, 1_{Q\times Q},
  \end{equation*}
  which means that $\tilde \Delta\, f$ is bounded. It is also clear that
  $\check f $ is continuous if $f$ is such, and that $\|\check
  f\|_\infty=\|f\|_\infty$ . This means that the maps
  in~\eqref{eq:Delta-epsilon-tilda-def} are well defined.

  To prove ($\widetilde H_1$), let $f\in \fc_b(Q)$ and consider
  \begin{equation*}
    \bigl( (\tilde \Delta \times \id)\circ \tilde \Delta\, f\bigr)(s_1, s_2, s_3)=
    \pair{\tilde \Delta\, f}{(\varepsilon_{s_1}*\varepsilon_{s_2})\otimes \varepsilon_{s_3}}
    =\pair{f}{\varepsilon_{s_1}* \varepsilon_{s_2} * \varepsilon_{s_3}}.
  \end{equation*}
  Similarly,
  \begin{equation*}
    \bigl((\id \times \tilde \Delta)\circ \tilde \Delta\, f\bigr) (s_1, s_2, s_3)=
    \pair{\tilde \Delta\, f}{\varepsilon_{s_1} \otimes (\varepsilon_{s_2}*\varepsilon_{s_3})}
    =\pair{f}{\varepsilon_{s_1}* \varepsilon_{s_2} * \varepsilon_{s_3}}.
  \end{equation*}

  It has already been mentioned, that all statements in ($\widetilde
  H_2$) except for (c) hold true. To prove (c) we use compactness of
  $F=\supp(\varepsilon_s*\varepsilon_t)$ in ($H_2$) a function $f\in
  \fcc_+(Q)$ such that $f(p)=1$ for all $p\in F$. Then
  \begin{align*}
    (\tilde \Delta \,f)(s,t)&=\pair{f}{\varepsilon_s*\varepsilon_t}=\int_Q
    f(p)\, d(\varepsilon_s*\varepsilon_t)(p)= \int_F f(p)\,
    d(\varepsilon_s*\varepsilon_t)(p)\\
    &=\int_F
    d(\varepsilon_s*\varepsilon_t)(p)=
    \pair{1_Q}{\varepsilon_s*\varepsilon_t}\\
    &=\pair{\tilde \Delta\, 1_Q}{\varepsilon_s\otimes \varepsilon_t}=
    1_{Q\times Q}(s,t)=1.
  \end{align*}

  \smallskip

  Let us prove ($\widetilde H_4$). Denote
  $E_0=\supp(\varepsilon_{s_0}*\varepsilon_{t_0})$. If $F$ is closed
  and $f_0$ is such that $\|f_0\|_\infty=1$, $f_0(r)=0$ for all $r\in
  F$ and $(\tilde \Delta\,
  f_0)(s_0,t_0)=\pair{f_0}{\varepsilon_{s_0}*\varepsilon_{t_0}}=1$,
  then it follows from Lemma~\ref{L:measure-support-separation} that
  $F\cap E_0=\emptyset$, hence $E_0$ is a subset of the open set
  $Q\setminus F$. Let $V\subset Q$ be an open neighborhood of $E_0$
  such that $\overline V \subset Q\setminus F$. By ($H_4$) there is an
  open neighborhood $U\subset Q\times Q$ of $(s_0,t_0)$ such that
  $\supp(\varepsilon_s*\varepsilon_t)\subset V$ for all $(s,t)\in
  U$. Let $f\in \fc_+(Q)$, $0\leq f(r)\leq 1$, be such that $f(r)=0$
  for all $r\in F$ and $f(r)=1$ for all $r\in \overline V$. Then, for
  any $(s,t)\in U$, we have
  \begin{equation*}
    (\tilde \Delta\, f)(s,t)=\pair{f}{\varepsilon_s*\varepsilon_t}=
    \pair{1_Q}{\varepsilon_s*\varepsilon_t}=1.
  \end{equation*}

  \smallskip

  It is easy to see that ($\widetilde H_5$) holds. Indeed. let $f\in
  \fc_b(Q)$ and $s\in Q$. Then
  \begin{equation*}
    \bigl( (\epsilon\times \id)\circ \tilde \Delta\, f\bigr)(s)=
    \pair{f}{\varepsilon_e*\varepsilon_s}=\pair{f}{\varepsilon_s}=f(s).
  \end{equation*}
  The second part of~\eqref{eq:epsilon-counit-property} is proved
  similarly.

  \smallskip

  Property ($\widetilde H_6$) is a direct consequence of ($H_6$).

  \smallskip Let us prove ($\widetilde H_7$). Let $f \in \fcc_+^e(G)$. Then
  \begin{equation*}
    (\tilde \Delta\, f)(s,\check s)=\pair{f}{\varepsilon_s*\varepsilon_{\check s}}>0,
  \end{equation*}
  since $e\in \supp(\varepsilon_s*\varepsilon_{\check s})$ by ($H_7$). Thus
  $(s,\check s)\in \supp(\tilde \Delta \, f)$ and
  \begin{equation*}
    \Gamma\,\check{}\, \subset
    \bigcap_{f\in \fcc_+^e(G)} S(\tilde \Delta\, f).
  \end{equation*}
  Conversely, let $s_0\neq t_0$, hence $(s_0,\check t_0)\notin
  \Gamma\,\check{}\,$. By ($H_7$), $e\notin
  \supp(\varepsilon_{s_0}*\varepsilon_{\check t_0})$. Since
  $\supp(\varepsilon_s*\varepsilon_{\check t})$ is compact there are
  open neighborhoods $U$ of $e$ and $V$ of $
  \supp(\varepsilon_s*\varepsilon_{\check t})$ such that $U\cap
  V=\emptyset$.  Let $f\in\fcc_+(U)$ be such that $f(e)=1$. Hence
  $f\in \fcc_+^e(G)$ and, as it follows from ($H_4$), there is an open
  neighborhood $W\subset Q\times Q$ of $(s_0,\check t_0)$ such that
  $\supp(\varepsilon_s*\varepsilon_{\check t}) \subset V$ for all
  $(s,\check t)\in W$, hence
  $\supp(f)\cap\supp(\varepsilon_s*\varepsilon_{\check t})=\emptyset$
  for such $(s,\check t)$. Thus
  \begin{equation*}
    (\tilde \Delta\, f)(s,\check t) =\pair{f}{\varepsilon_s*\varepsilon_{\check t}}=0,
  \end{equation*}
  and $(s_0,\check t_0)\notin \supp(\tilde \Delta\, f)$. This means that
  \begin{equation*}
    (s,\check t)\notin \bigcap_{f\in\fcc_+^e(G)} \supp(\tilde \Delta\, f),
    \quad (s,\check t) \in W,
  \end{equation*}
  which finishes the proof.
\end{proof}

\begin{remark}
  It immediately follows from ($\widetilde H_2$)~(a) and ($\widetilde
  H_2$)~(b) that ($\widetilde H_2$)~(c) is equivalent to the condition
  that there is $f\in\fcc_+(Q)$ such that $\tilde \Delta\, f=\|\tilde \Delta\|$.
\end{remark}

\begin{proposition}\label{P:C^*-algebra->hypergroup}
  Let $Q$ be a Hausdorff locally compact space, and let a linear map
  $\tilde \Delta\colon \fc_b(Q) \rightarrow \fc_b(Q\times Q)$,
  $*$-homomorphisms $\,\check{}\,\colon \fc_b(Q) \rightarrow \fc_b(Q)$
  and $\epsilon\colon \fc_b(Q) \rightarrow \Bbb C$ be given and satisfy
  the properties \textup{(}$\widetilde
  H_1$\textup{)}~--~\textup{(}$\widetilde H_7$\textup{)} in
  Proposition~\ref{P:hypergroup->C^*-algebra}. For $\mu_1,\mu_2,\mu
  \in \meas_b(Q)$, define $\mu_1*\mu_2\in \meas_b(Q)$ and $\check
  \mu\in \meas_b(Q)$ by
  \begin{equation}\label{eq:b-measure-convolution-def}
    \begin{array}{rcl}
      \pair{f}{\mu_1*\mu_2}&=&
            \displaystyle
            \int_{Q\times Q} (\tilde \Delta f)(s_1,s_2)\,
      d\mu_1(s_1) d \mu_2(s_2),
      \\[4mm]
      \pair{f}{\check \mu}&=&\pair{\check f}{\mu},

    \end{array}
  \end{equation}
  where $f \in \fcc(Q)$, and let $e=\supp(\epsilon)$. Then
  $(Q,*,e,\,\check{}\,)$ is a locally compact hypergroup.
\end{proposition}

\begin{proof}
  First consider ($H_1$) and show that $(\meas_b(Q),*)$ is indeed an
  associative algebra, if the convolution of measures is given by the
  first formula in~\eqref{eq:b-measure-convolution-def}.

  If $\mu_1, \mu_2 \in\meas_b(Q)$, then the measures $|\mu_1|$ and
  $|\mu_2|$ are also bounded. Moreover,
  by~\eqref{eq:b-measure-convolution-def} and using ($\widetilde
  H_2$), we have
  \begin{align*}
    |\pair{f}{\mu_1*\mu_2}|&=\bigl|
    \int_{Q\times Q} (\tilde \Delta\, f)(s_1, s_2)\, d\mu_1(s_1) d \mu_2(s_2)\bigr|\\
    &\leq \int_{Q\times Q} (\tilde \Delta \bigl| f|)(s_1,s_2)\, d|\mu_1|(s_1) d|\mu_2|(s_2)\\
    &\leq \int_{Q\times Q} \tilde \Delta\bigl(\|f\|_\infty\, 1_Q\bigr)(s_1,s_2)\,
    d|\mu_1|(s_1) d|m_2|(s_2)\\
    &=\|f\|_\infty \int_{Q\times Q} 1_Q(s_1)1_Q(s_2)\, d|\mu_1|(s_1) d|\mu_2|(s_2)\\
    &=\|f\|_\infty\, \|\mu_1\|\, \|\mu_2\|,
  \end{align*}
  which shows that $\|\mu_1*\mu_2\|\leq \|\mu_1\|\, \|\mu_2\|$, hence
  $\mu_1*\mu_2 \in \meas_b(Q)$.

  Associativity of the algebra $(\meas_b(Q),*)$ is immediately implied
  by ($\widetilde H_1$).

  \smallskip Consider ($H_2$). Since $\tilde \Delta$ is positive by
  ($\widetilde H_2$), the measure $\varepsilon_s*\varepsilon_t$ is
  nonnegative for $s,t \in Q$. Moreover,
  \begin{equation*}
    \pair{1_Q}{\varepsilon_s*\varepsilon_t}=
    \pair{1_Q\otimes 1_Q}{\varepsilon_s\otimes \varepsilon_t}=1,
  \end{equation*}
  hence $\varepsilon_s *\varepsilon_t \in \meas_p(Q)$.

  Let us show that $\supp(\varepsilon_s*\varepsilon_t)$ is compact for
  any $s,t\in Q$. Indeed, using ($\widetilde H_2$)~(c) there is $f\in
  \fcc_+(Q)$ such that $\|f\|_\infty\leq 1$ and $(\tilde \Delta\,
  f)(s,t)=1$. This means that
  $\pair{f}{\varepsilon_s*\varepsilon_t}=1$. Denote $F=\supp(f)$ and
  let $g\in \fcc_+(Q)$ be arbitrary such that $\|g\|_\infty\leq 1$ and
  $\supp(g) \subset Q\setminus F$. Then $f+g \leq 1_Q$, that is,
  $g\leq 1_Q-f$, and
  \begin{equation*}
    0\leq \pair{g}{\varepsilon_s*\varepsilon_t}\leq
    \pair{1_Q-f}{\varepsilon_s*\varepsilon_t}=
    \pair{1_Q}{\varepsilon_s*\varepsilon_t}-\pair{f}{\varepsilon_s*\varepsilon_t}
    =0.
  \end{equation*}
  This shows that $\pair{g}{\varepsilon_s*\varepsilon_t}=0$ for any
  such $g$, which means that $\supp(\varepsilon_s*\varepsilon_t)
  \subset F$. Since $F$ is compact and support of a measure is closed,
  $\supp(\varepsilon_s*\varepsilon_t)$ is compact.

  \smallskip

  Property ($H_3$) is immediate, since $\tilde \Delta\, f\in
  \fc_b(Q\times Q)$ by definition for any $f\in \fc_0(Q)$.

  \smallskip

  Consider ($H_4$). Let $(s_0,t_0)\in Q\times Q$, and
  denote $F_0=\supp(\varepsilon_{s_0}*\varepsilon_{t_0})$, which is a
  compact set by ($H_2$). Let $V,W\subset Q$ be open such that
  $\ucomps{V}{W}$ is an open neighborhood of $F_0$ in the Michael
  topology, that is, $F_0\subset W$ and $F_0\cap V\neq
  \emptyset$. Choose $p\in F_0\cap V$ and let $f\in
  \fcc_+^p(V)$. Then, since $p\in F_0$,
  \begin{equation*}
    (\tilde \Delta\, f)(s_0,t_0)=\pair{f}{\varepsilon_{s_0}*\varepsilon_{t_0}}>0,
  \end{equation*}
  and, by continuity of $\tilde \Delta\, f$, we can find a neighborhood
  $U_1\subset Q\times Q$ of $(s_0,t_0)$ such that $(\tilde \Delta\,
  f)(s,t)>0$ for all $(s,t)\in U_1$. This would imply that
  $\pair{f}{\varepsilon_s*\varepsilon_t}>0$, hence $\supp(f)\cap
  \supp(\varepsilon_s*\varepsilon_t)\neq \emptyset$, in particular,
  $V\cap\supp(\varepsilon_s*\varepsilon_t)\neq \emptyset$ for all
  $(s,t) \in U_1$.

  To proceed, let $W\subset Q$ be open and
  $\supp(\varepsilon_{s_0}*\varepsilon_{t_0})\subset W$. Let
  $F=Q\setminus \supp(\varepsilon_{s_0}*\varepsilon_{t_0})$. Then
  there is a function $f_0\in \fc_+(Q)$, $0\leq f(r)\leq 1$, $r\in Q$,
  such that $f(r)=0$ for $r\in F$ and $f(r)=1$ for $r\in
  \supp(\varepsilon_{s_0}* \varepsilon_{t_0})$. We have
  $\pair{f_0}{\varepsilon_{s_0}*\varepsilon_{t_0}}=
  \pair{1_Q}{\varepsilon_{s_0}*\varepsilon_{t_0}}
  =1=\|\varepsilon_{s_0}*\varepsilon_{t_0}\|$. Hence, there is a
  neighborhood $U\subset Q\times Q$ of the point $(s_0,t_0)$ and a
  function $f\in \fc_+$, $\|f\|_\infty=1$, such that $f(r)=0$ for all
  $r\in F$ and $\pair{f}{\varepsilon_s*\varepsilon_t}=1$ for all
  $(s,t)\in U$. By Lemma~\ref{L:measure-support-separation},
  $F\cap\supp(\varepsilon_s*\varepsilon_t)=\emptyset$, hence
  $\supp(\varepsilon_s*\varepsilon_t)\subset W$ for all $(s,t) \in U$.

  \smallskip

  ($H_5$) (resp., ($H_6$)) is immediately implied by
  ($\widetilde H_5$) (resp., ($\widetilde H_6$).

  \smallskip

  Let us prove ($H_7$). Let $s\neq t$ and show that $e
  \notin \supp(\varepsilon_s*\varepsilon_{\check t})$. Indeed,
  $(s,\check t)\notin \Gamma\,\check{}\,$, hence there exists $f\in
  \fcc_+^e(G)$ such that $(s,\check t)\notin \supp(\tilde \Delta\, f)$. This
  means that
  \begin{equation*}
    (\tilde \Delta \, f)(s,\check t)=\pair{f}{\varepsilon_s*\varepsilon_{\check t}}=0.
  \end{equation*}
  This means that
  $f\restriction_{\supp(\varepsilon_s*\varepsilon_{\check t})}=0$ but
  $f(e)>0$. Hence $e\notin \supp(\varepsilon_s*\varepsilon_{\check
    t})$.

  Conversely, let us show that $e\in
  \supp(\varepsilon_s*\varepsilon_{\check s})$ for any $s\in
  Q$. Indeed, if $U$ is an arbitrary open neighborhood of $e$, choose
  a function $f\in\fcc_+^e(U)$. Since $(s,\check s) \in
  \Gamma\,\check{}\,$, it follows from~\eqref{eq:graph-involution}
  that $(s,\check s)\in \supp(f)$, which means that
  \begin{equation*}
    (\tilde \Delta\, f)(s,\check s)=\pair{f}{\varepsilon_s*\varepsilon_{\check s}}>0,
  \end{equation*}
  and $e\in \supp(\varepsilon_s*\varepsilon_{\check s})$.
\end{proof}

\subsection{Fourier and Fourier-Stieltjes algebras}
\label{Fourier_alg}

\begin{definition}\label{def_repr}
  Let $Q$ be a locally compact hypergroup and $H$ a Hilbert space. A
  representation $\pi$ of the hypergroup $Q$ on the Hilbert space $H$
  is an involution homomorphism $\pi: \meas_b(Q) \to B(H)$ of the
  involution algebra $(\meas_b(Q), *, {}^*)$ into the $C^\ast$-algebra
  $B(H)$ of all linear bounded operators on $H$. The set of all
  representations of $Q$ will be denoted by $\Sigma$.
\end{definition}

Let us remark that definition~\ref{def_repr} can be rewritten in terms
of the comultiplication in~(\ref{eq:Delta-epsilon-tilda-def}) as
follows~\cite{Ch_Ka_Po2010}: a weakly continuous map $\pi : Q\to B(H)$
is ca;;ed a representation of the hypergroup~$Q$ if the following
conditions are satisfied:
\begin{itemize}
\item [(i)] $\pi(e) = I$, where $I$ is the identity operator on $H$;
\item [(ii)] $\pi(p)^\ast = \pi (p^\ast)$ for all $p\in Q$;
\item [(iii)] for any $\xi, \eta\in H$, we have $\tilde\Delta
  (\pi(\cdot)\xi, \eta)_H (p, q)= (\pi(p)\pi(q)\xi, \eta)_H$.
\end{itemize}

\begin{definition}\label{left_reg_repr}
  Let $\tilde m$ be a left Haar measure on the hypergroup~$Q$. A left
  (resp., right) regular representation is a mapping $\lambda : Q\ni p
  \mapsto L_p\in B(L_2(Q))$ (resp., $\mu : Q\ni p \mapsto R_p \in
  B(L^2(Q))$) defined, for $f \in \fcc(Q)$, by $(L_p f)(q) =
  (\tilde\Delta f)(p^\ast, q)$ (resp., $(R_p f)(q) = (\tilde\Delta
  f)(p, q)\kappa^{\frac 1 2} (p)$) and then extended to $L^2(Q)$ by
  continuity in virtue of the estimates $\|L_p f\|_2 \leq \|f\|_2$ and
  $\| R_p f \|_2\leq \|f\|_2$ \cite{Ch_Ka_Po2010}.
\end{definition}

Let us recall that each representation of a hypergroup can be uniquely
continued to a nondegenerate representation of the Banach $*$-algebra
$L^1(Q)$ and, conversely, each nondegenerate representation of the
$\ast$-algebra $L^1 (Q)$ gives rise to a representation of the
hypergroup $Q$ \cite{Ch_Ka_Po2010}. By $C^\ast (Q)$, we denote the
full $C^\ast$-algebra of the hypergroup $Q$, that is the closure of
$L^1(Q)$ with respect to the $C^*$-norm $\|f\| = \mbox{sup}_{\pi\in
  \Sigma} \pi (f)$.

The reduced $C^*$-algebra of the hypergroup $Q$ is denoted by
$C^\ast_r (Q)$ that is the $C^*$-closure of the $*$-algebra generated
by the family $\{L_p: p \in Q\}$ of operators on $L^2(Q)$ with respect
to the norm of $B(L^2(Q))$.

\begin{definition}\label{def_Four_St_alg}
  The Banach space dual to the full $C^*$-algebra $C^*(Q)$ will be
  called the Fourier-Stieltjes space and denoted by $\FSs(Q)$. The
  Banach space dual to the reduced $C^*$-algebra of the hypergroup
  $Q$, $C^*_r(Q)$, will be denoted by $\FSs_\lambda(Q)$.
\end{definition}

\begin{definition}
  Let $(H_1,\pi_1)$ and $(H_2,\pi_2)$ be two representations of a
  hypergroup $Q$. The representation $\pi_1$ is said to be weakly
  contained in the representation $\pi_2$ if the kernel of the
  representation $\pi_1$ contains the kernel of the representation
  $\pi_2$.
\end{definition}

It is well known, see~\cite{muruganandam07:_fourier_algeb} that for
any $\alpha\in \FSs (Q)$ (respectively, $\alpha\in \FSs_\lambda (Q)$)
there is a representation $(H, \pi)$ of the hypergroup $Q$
(respectively, a representation $(H, \pi)$ weakly contained in the
left regular representation) and two vectors $\xi, \eta\in H$ such
that the function $a\in C_b (Q)$ given by
\begin{equation*}
  a(p) = (\pi(p)\xi, \eta)_H ,\quad  p\in Q,
\end{equation*}
defines a linear functional $\alpha$ by
\begin{equation*}
\alpha (f) = \int_Q a(p) f(p) dm(p), \quad f\in L_1(Q),
\end{equation*}
such that its norm is given by
\begin{equation*}
  \|\alpha\| = \sup_{f\in L_1 (Q), \|f\|_\Sigma' = 1} \left | \int_Q a(p) f(p) dm(p) \right | = \|\xi\|_H \|\eta\|_H,
\end{equation*}
where $\Sigma' = \Sigma$ (respectively, $\Sigma' = \lambda$).

\begin{definition}\label{def_Four_alg}
  The Fourier space of a hypergroup $Q$ will be called the closure of
  the space generated by the elements $f\ast f^\dagger$, $f\in \fcc
  (Q)$ with respect to the norm of the space $\FSs_\lambda (Q)$, and
  denoted by $\Fs (Q)$.
\end{definition}

\subsection{Tensor product and conditional expectation}

Let $A$ and $B$ be $C^*$-algebras. The tensor product of $A$ and $B$,
denoted by $A\otimes B$, is the completion of the algebraic tensor
product $A\odot B$ with respect to the min-$C^*$-norm on $A\odot B$,
\begin{equation*}
  \|\sum_{i=1}^n a_i \otimes b_i\|_{\textup{min}}=\sup_{\pi_A \in \Sigma_A, \pi_B \in \Sigma_B} \Bigl\|\sum_{i=1}^n \pi_A(a_i)\otimes \pi_B(b_i)\Bigr\|  ,
\end{equation*}
where $\Sigma_A$ (resp., $\Sigma_B$) is the set of all representations
of the $C^*$-algebra $A$ (resp., $B$)~\cite{takesaki03_i:_theor}.

For a $C^*$-algebra $A$, the $C^*$-algebra of multipliers of $A$ is
denoted by $M(A)$, see~\cite{takesaki03_i:_theor} for details.

\medskip

Let $A$ be a $C^*$-algebra and $B \subset A$ a $C^*$-subalgebra of
$A$. A bounded linear map $P\colon A \rightarrow B$ is called a
{\sl conditional expectation} if it satisfies the following
properties~\cite{umegaki54:_condit_i}:
\begin{itemize}
\item [(i)] $P$ is a projection onto and has norm $1$, that is,
  $P^2=P$ and $\|P\|=1$;
\item [(ii)] $P$ is positive, that is $P(a^* a) \geq 0$ for any $a\in
  A$;
\item [(iii)] $P(b_1 a b_2)=b_1 P(a) b_2$ for any $a \in A$ and $b_1,
  b_2 \in B$;
\item [(iv)] $P(a^*) P(a) \leq P(a^*a)$ for all $a\in A$.
\end{itemize}
It follows from (ii) and the polarization identity
that
\begin{itemize}
\item[(v)]
$P(a^*)=P(a)^*$, $a\in A$.
\end{itemize}

It also follows from~\cite{tamiyama57:_w} that (i) implies (ii), (iii), and (iv).

\section{Main results}\label{sec:main-results}

Let $G$ be a locally compact group with a Haar measure $m$, $A$ denote
the $C^*$-algebra $\fc_b(G)$, $A_0$ its $C^*$-subalgebra $\fc_0 (G)$,
and let $\fcc(G)$ be the ideal of $A$ consisting of functions with
compact supports. Let $P\colon A \rightarrow A$ be a conditional
expectation. The restriction of $P$ to $\fc_0(G)$ is still denoted by
$P$.

\begin{proposition}\label{P:preparatory}
  Let $B$ be a $C^*$-subalgebra of $\fc_0(G)$, $P\colon \fc_0(G)
  \rightarrow B$ be a conditional expectation. Let $Q$ denote the
  spectrum of the commutative $C^*$-algebra $B$, and identify $B$ with
  $\fc_0(Q)$ via the Gelfand transform. Let $\iota\colon \fc_0(Q)
  \cong B \rightarrow A$ denote the embedding map. Then we have the following.
  \begin{enumerate}
  \item[\textup{(i)}] There is a unique continuous surjection
    $\varphi\colon G \rightarrow Q$ such that $\iota(\tilde
    f)(p)=\tilde f(\varphi(p))$ for any $\tilde f\in \fc_0 (Q)$ and
    any $p\in G$. If $p, q \in G$ are such that
    $\varphi(p)=\varphi(q)$, then $\iota (\tilde f)(p)=\iota(\tilde
    f)(q)$.
  \item[\textup{(ii)}] Denote by $P^*\colon \meas_b(Q) \rightarrow
    \meas_b(G)$ the linear map adjoint to $P$,
    \begin{equation*}
      \pair{f}{P^*(\tilde \mu)}=\pair{P(f)}{\tilde \mu},
      \qquad \tilde \mu \in \meas_b(Q),\quad f\in \fc_b(G),
    \end{equation*}
    and let $\varphi_*\colon \meas_b(G) \rightarrow \meas_b(Q)$ be the
    linear map induced by $\varphi$,
    \begin{equation*}
      \pair{\tilde f}{\varphi_*(\mu)}=\pair{\tilde f\circ \varphi}{\mu},
      \qquad
      \tilde f\in\fc_b(Q),
      \quad \mu \in \meas_b(G).
    \end{equation*}
    Then
    \begin{equation*}
      \varphi_*\circ P^*=\id_{\meas_b(Q)}.
    \end{equation*}
  \item [\textup{(iii)}]\label{P:item:compactness} Let $s\in Q$, and
    denote
    \begin{equation*}
      O_s=\{p\in G: \varphi(p)=s\}.
    \end{equation*}
    If $P(\fcc_+(G))\subset \fcc_+(G)$ and, for $p\in G$ and $f\in
    \fcc_+(G)$, $P(f)(\varphi(p))>0$ whenever $f(p)>0$, then $O_s$ is
    compact.
  \item [\textup{(iv)}] For any $s\in Q$,
    \begin{equation*}
      \supp(P^*(\varepsilon_s))\subset O_s.
    \end{equation*}
    If $P(f)(\varphi(p))>0$ whenever $f(p)>0$ for $p\in G$ and $f\in
    \fcc_+(G)$, then
    \begin{equation*}
      \supp(P^*(\varepsilon_s))= O_s.
    \end{equation*}
  \end{enumerate}
\end{proposition}

\begin{proof}
  (\emph{i}) Regarding $G$ as $\spec(\fc_0(G))$, the spectrum of the
  commutative $C^*$-algebra $\fc_0(G)$, we define
  \begin{equation*}
    \varphi(p)(\tilde
    f)=p\!\!\upharpoonright_B(\tilde f)
  \end{equation*}
   for $\tilde f \in B$ and $p\in
  \spec(\fc_0(G))$. Hence, by the definition, for $\tilde f \in
  \fc_0(Q)$ and $p\in G$, we have
  \begin{equation*}
    \iota(\tilde f)(p)=\tilde f(\varphi(p))
    =p\!\!\upharpoonright_B(\tilde f).
  \end{equation*}
  If $s\in Q=\spec(B)$, then it follows from~\cite{Naimark} that $s$
  can be extended to an element $p\in \spec(\fc_0(G))$ such that
  $p\!\!\upharpoonright_B=s$, hence $\varphi$ is onto.

  The last statement in the item is a direct implication of the
  definition of $\varphi$.

  \medskip
  \noindent  (\emph{ii}) If $\tilde f \in \fc_0(Q)$, then
  $(P\circ\iota)(\tilde f) =\tilde f$, since $P$ is a
  projection. Hence, for $\tilde f\in \fc_0(Q)$ and $\tilde \mu \in
  \meas_b(Q)$, we have
  \begin{equation*}
    \pair{\tilde f}{(\varphi_*\circ P^*)(\tilde \mu)}
    =\pair{\tilde f\circ \varphi}{P^*(\tilde \mu)}
    =\pair{\iota(\tilde f)}{P^*(\tilde \mu)}
    =\pair{(P\circ \iota)(\tilde f)}{\tilde \mu}=
    \pair{\tilde f}{\tilde \mu}.
  \end{equation*}

  \smallskip
  \noindent (\emph{iii}) Indeed, for $p\in O_s$, let $f\in
  \fcc_+(G)$ such that $f(p)=1$. Then, by the assumption,
  $P(f)(s)=P(f)(\varphi(p))>0$. Hence, for any $q\in O_s$,
  $P(f)(q)=P(f)(s)>0$, and $O_s\subset \supp(P(f))$. Since
  $O_s=\varphi^{-1}(s)$ is closed and $\supp P(f)$ is compact, $O_s$
  is compact.

  \medskip
  \noindent (\emph{iv})
  We first prove that $\supp(P^*(\varepsilon_s))\subset O_s$. Let $p\in
  \supp(P^*(\varepsilon_s))$ and assume that $p\notin O_s$, hence
  $\varphi(p)\neq s$. Choose an open neighborhood $V$ of $\varphi(p)$
  such $s\notin V$, and let $U=\varphi^{-1}(V)$. Then $p\in U$ and,
  for any $f\in \fcc_+(U)$, $\supp(P(f))\subset V$, hence
  $P(f)(s)=0$. Thus
  \begin{equation*}
    \pair{f}{P^*(\varepsilon_s)}
    =\pair{P(f)}{\varepsilon_s}=0,
  \end{equation*}
  and $p\notin \supp(P^*(\varepsilon_s))$.

  It remains to prove, with the assumption made, that $O_s\subset
  \supp(P^*(\varepsilon_s))$. Let $p\in O_s$ and $U$ be an open
  neighborhood of $p$. Choose a function $f\in \fcc_+(U)$ such that
  $f(p)=1$. By the assumption, $P(f)(s)>0$. Hence
  \begin{equation*}
    \pair{f}{P^*(\varepsilon_s)}
    =\pair{P(f)}{\varepsilon_s}
    =P(f)(s)>0,
  \end{equation*}
  and $p\in \supp P^*(\varepsilon_s)$.
\end{proof}

\begin{remark}
  In the sequel, we will identify elements of $\fc(Q)$ with those in
  $B$, according to Proposition~\ref{P:preparatory}~(i).
\end{remark}

\begin{theorem}\label{Th:hypergroup_from_expectation}
  Let $G$ be a locally compact group, $P\colon \fc_b(G)\rightarrow
  \fc_b(G)$ be a conditional expectation. Assume that $P$ satisfies
  the assumptions in Proposition~\ref{P:preparatory}~(iii), and let
  the following hold:
  \begin{equation}
    \label{eq:conditions_P}
    \begin{array}{c}
      \displaystyle
      \bigl((P\times \id)\circ \Delta \circ P\bigr)(f)
      =
      \bigl((\id \times P)\circ \Delta \circ P\bigr)(f)
      = \bigl((P\times P)\circ \Delta \bigr)(f),
      \\[2mm]
      P(\check f) = \bigl(P(f)\bigr)\check{\strut},
    \end{array}
  \end{equation}
  for all $f \in \fc_b(G)$, where $\Delta\colon\fc_b(G) \rightarrow
  \fc_b(G\times G)$ is the group comultiplication on $\fc_b(G)$
  defined by $(\Delta f)(p_1, p_2) = f(p_1 p_2)$,  and
  $\check f(p)=f(p^{-1})$, $p_1, p_2, p\in G$.

  Let $B=P(\fc_0(G))$, denote by $Q$ the spectrum of the
  commutative $C^*$-algebra $B$, and identify $B$ with
  $\fc_0(Q)$.
  Define $\tilde \Delta \colon \fc_b(Q) \rightarrow \fc_b(Q\times Q)$
  by
  \begin{equation*}
    \tilde \Delta (\tilde f) = \bigl((P \times P)\circ \Delta\bigr) (\tilde f),
    \quad \tilde f\in \fc_b(Q).
  \end{equation*}
  For $s\in Q$, define $\check s\in Q$ by
  \begin{equation*}
    \check s(\tilde f)= \check {\tilde f}(s),\quad \tilde f \in B.
  \end{equation*}
  Let $\tilde e\in Q$ be defined by
  \begin{equation*}
    \tilde e(\tilde f)=\tilde f(e),
    \quad \tilde f \in B,
  \end{equation*}
  where $e$ is the identity of the group $G$.

  Then $(Q, \tilde e,\, \check{}\,)$ is a locally compact hypergroup
  with comultiplication $\tilde \Delta$.

  If $m\circ P=m$, where $m$ is a left Haar measure on the group
  $G$, then $\tilde m\in \meas(Q)$ defined by $\tilde m
  =\varphi_*(m)=m\!\!\restriction_B$ is a left Haar measure on $Q$.
\end{theorem}

\begin{proof}
  To prove the theorem, using
  Proposition~\ref{P:C^*-algebra->hypergroup}, we will check the
  conditions ($\tilde H_1$)---($\tilde H_7$) in
  Proposition~\ref{P:hypergroup->C^*-algebra}.

  \smallskip

  ($\tilde H_1$) We have
  \begin{align*}
    (\tilde \Delta \times \id)\circ \tilde \Delta&=
    (P\times P\times \id) \circ (\Delta \times \id) \circ (P \times P)
    \circ \Delta\\
    &=(P\times P\times P)\circ (\Delta \times \id) \circ \Delta.
  \end{align*}
  On the other hand,
  \begin{align*}
    (\id \times \tilde \Delta)\circ \tilde \Delta&=
    (\id \times P \times P) \circ (\id \times \Delta ) \circ
    (P\times P)\circ \Delta\\
    &=(P\times P\times P)\circ (\id \times \Delta)\circ \Delta.
  \end{align*}
  Now, the result follows from coassociativity of $\Delta$.

  \smallskip

  ($\tilde H_2$)~(b)
  Since $P(1_G)=1_Q$ and $\Delta 1_G=1_{G\times G}$, the claim is immediate.

  ($\tilde H_2$)~(a) It is clear that $\tilde \Delta$ is
  positive, being a composition of the positive maps $\Delta$ and $P\times P$.
  It is also immediate from positivity and ($\tilde H_2$)~(b) that
  $\tilde \Delta(\tilde f) \in \fc_{b,+}(Q\times Q)$ for $\tilde f \in
  \fc_{b,+}(Q)$.

  ($\tilde H_2$)~(c) For any $\tilde f \in \fcc_+(Q)$ and $s,t\in Q$, we have
  \begin{align*}
    (\tilde \Delta\, f)(s,t)
    &=\pair{\tilde \Delta\, \tilde f}{\varepsilon_s\otimes \varepsilon_t}
    =\pair{\bigl((P\times P)\circ \Delta\bigr)(\tilde f)}
    {\varepsilon_s \otimes \varepsilon_t}\\
    &=\pair{\Delta\, \tilde f}{P^*(\varepsilon_s)\otimes P^*(\varepsilon_t)}
    =\pair{\tilde f}{P^*(\varepsilon_s)*_{G} P^*(\varepsilon_t)},
  \end{align*}
  where $*_G$ is the convolution of measures on $G$ defined by
  \begin{equation*}
    \pair{f}{\mu *_{G} \nu}=\int_{G^2} f(pq)\, d\mu(p)d\nu(q),
    \qquad f\in \fc_b(G),
    \quad
    \mu,\nu \in \meas_c(G).
  \end{equation*}
  It follows from Proposition~\ref{P:preparatory} that
  $P^*(\varepsilon_s)$ and $P^*(\varepsilon_t)$ are compact, hence the
  set $F=\supp(P^*(\varepsilon_s)*_G P^*(\varepsilon_t))$ is also
  compact. It also follows from the first identity
  in~(\ref{eq:conditions_P}) that, if $\tilde f=P(f)$, $f\in
  \fc_b(G)$, then
  \begin{equation*}
    \pair{\tilde f}{P^*(\varepsilon_s)*_GP^*(\varepsilon_t)}
    =\pair{f}{P^*(\varepsilon_s)*_GP^*(\varepsilon_t)}.
  \end{equation*}
  Hence, let $f\in\fcc_+(G)$ be such that $f(p)=1$ for all $p \in
  \supp(P^*(\varepsilon_s)*_GP^*(\varepsilon_t))$. Then, by the
  assumption, $\tilde f=P(f)\in \fcc_+(Q)$, and
  \begin{align*}
    \pair{\tilde f}{P^*(\varepsilon_s)*_GP^*(\varepsilon_t)}
    &=\pair{ f}{P^*(\varepsilon_s)*_GP^*(\varepsilon_t)}
    =\pair{1_G}{P^*(\varepsilon_s)*_GP^*(\varepsilon_t)}\\
    &=\pair{1_{Q\times Q}}{\varepsilon_s\otimes \varepsilon_t}=1.
  \end{align*}

  \smallskip ($\tilde H_4$) Let $F\subset Q$ be closed, $(s_0,t_0)\in
  Q\times Q$, and $\tilde f_0 \in \fcc_+(Q)$ be such that $\|\tilde
  f_0\|\leq 1$, $f_0(r)=0$ for $r\in F$ and $(\tilde \Delta \, \tilde
  f_0)(s_0,t_0)=1$. Since $\tilde \Delta \, \tilde f_0\in
  \fc_+(Q\times Q)$ there is a neighborhood $U\subset Q\times Q$ of
  the point $s_0,t_0$ such that $(\tilde \Delta\, \tilde
  f_0)(s,t)>\frac{1}{2}$ for all $(s,t)\in U$. This means that
  \begin{equation*}
    \pair{\tilde f_0}{P^*(\varepsilon_s)*_G P^*(\varepsilon_t)} >\frac{1}{2},
    \quad
    (s,t)\in U.
  \end{equation*}
  This means that the set $E$,
  \begin{equation*}
    E=\bigcup_{(s,t)\in U}\supp(P^*(\varepsilon_s)*_G P^*(\varepsilon_t))\subset
    \supp \tilde f_0,
  \end{equation*}
  is compact. Choose $\tilde f\in \fcc_+(Q)$ to be such that $\tilde
  f(r)=1$ if $r\in E$, and $\tilde f(r)=0$ if $r\in F$. This function
  satisfies the condition in ($\tilde H_4$).

  \smallskip

  ($\tilde H_5$)
  For $\tilde f \in\fc_0(Q)$, we have
  \begin{align*}
    (\epsilon \times\id) \circ \tilde \Delta (\tilde f)&=
    (\epsilon \times \id) \circ (P\times P) \circ \Delta (P(\tilde f))\\
    &=
    (\epsilon \times P)\circ \Delta (P(\tilde f))=P(P(\tilde f))=\tilde f.
  \end{align*}

  \smallskip

  ($\tilde H_6$) First of all note that $\check {\tilde f}
  \in \fcc(Q)$ for $\tilde f \in \fcc(Q)$, since $P(\check {\tilde
    f})=P(\tilde f)\,\check\strut=\check {\tilde f}$
  by~(\ref{eq:conditions_P}).

  It is clear that the first identity
  in~(\ref{eq:involution-property}) holds. To prove the second
  identity in~(\ref{eq:involution-property}), consider
  \begin{multline*}
    (\tilde \Delta \check {\tilde f} )(s,t)= \bigl((P\times P)\circ \Delta
    \check {\tilde f}\bigr)(s,t)= \bigl((P\times P)\circ\Sigma\circ
    (\ \check{}\  \times\ \check{}\ )\circ \Delta \bigr)(\tilde f)(s,t)
    \\
    = \bigl((P\times P)\circ \Delta \bigr)(\tilde f)(\check t,\check s)= (\tilde \Delta\,
    \tilde f)(\check t,\check s)
  \end{multline*}
  for $s,t \in  Q$, $\tilde f \in B$.

  \smallskip ($\tilde H_7$) We will prove that ($H_7$) holds true. To
  this end, it is sufficient to prove that $e\in O_s\cdot O_{\check
    t}$, $s,t\in Q$, if and only if $s= t$, since the set $O_s\cdot
  O_{\check t}$ is compact because such are $O_s$ and $O_{\check
    t}$. Since $O_{\check t}=O_t^{-1}$, $p^{-1}\in O_{\check s}$ for
  $p\in O_s$, hence $e\in O_s\cdot O_{\check s}$. On the other hand,
  if $e\in O_s\cdot O_{\check t}$, then there is $p\in O_s\cap O_t$,
  that is, $s=t$.

  \smallskip

  Let us prove that $\tilde m = \varphi_*(m)$ is a left
  invariant measure on $Q$, that is, for any $\tilde f, \tilde g\in \fcc(Q)$ the
  following relation holds:
  \begin{equation}\label{eq:Haar_measure_def}
    (\id \times \tilde m)\bigl( \tilde \Delta (\tilde f)\cdot (1\otimes
    \tilde g)\bigr)
    = (\id\times \tilde m)\bigl( (1\otimes \tilde f)\cdot
    (\, \check{}\,\times \id)\circ  \tilde \Delta (\tilde g)\bigr).
  \end{equation}
  Consider the left-hand side of~\eqref{eq:Haar_measure_def}, use the
  definitions of $\tilde m$ and $\tilde \Delta$, and let $\tilde f=P
  f$, $\tilde g=Pg$, $f,g \in \fcc(G)$,
  \begin{multline*}
    (\id \times m) \bigl(((P\times P) \circ \Delta (Pf)) \cdot
    (1\otimes Pg)\bigr)
    \\
    = (\id \times m)\bigl((P \times \id)\circ \Delta (Pf) \cdot
    (1\otimes Pg)\bigr)
    \\
    = P\bigl( (\id \times m)( \Delta (Pf) \cdot
    (1\otimes Pg))\bigl).
  \end{multline*}
  Now using a relation similar to (\ref{eq:Haar_measure_def}) that
  holds for $m$ on $G$ and that $P\circ\,\check{}\,= \check{}\,\circ P $
  we get
  \begin{multline*}
    P\bigl( (\id \times m)( \Delta (Pf) \cdot (1\otimes Pg))\bigl)
    \\
    =P\bigl((\id \times m)((1\otimes Pf)\cdot (\,\check{}\, \times \id)\circ
    \Delta (Pg))\bigr)
    \\
    = (\id \times m)\bigl((1\otimes Pf)\cdot (\,\check{}\, \circ P\times
    \id)\circ \Delta (Pg)\bigr)
    \\
    = (\id \times m)\bigl((1\otimes Pf) \cdot (\,\check{}\, \circ P \times
    P)\circ \Delta (Pg)\bigr)
    \\
    = (\id \times m\circ P) \bigl(
    (1\otimes Pf)\cdot (\,\check{}\, \times P)\circ \Delta (P g)\bigr)
    \\
    = (\id \times \tilde m) \bigl( (1\otimes Pf)\cdot (\,\check{}\, \times
    \id)\circ \tilde \Delta (Pg)\bigr).
  \end{multline*}

\end{proof}

\begin{lemma}
  Let $\pair{ P f}{m}=\pair{f}{m}$ for all $f\in \fcc(G)$. Then the
  conditional expectation $P\colon \fcc(G) \rightarrow B$ can be
  extended by continuity to an idempotent on $L_1(G,m)$ and an
  orthogonal projection on $L_2(G, m)$, still denoted by $P$.
\end{lemma}

\begin{proof}
  Let $f\in \fcc(G)$. Since $-|f|\leq f \leq |f|$ and $P$ preserves
  the cone of nonnegative functions, we have $|P(f)|\leq P(|f|)$. Hence,
  \begin{align*}
    \|P(f)\|_1&=\int_G|P(f)|(p)\, dm(p)\leq \int_GP(|f|)(p)\, dm(p)\\
    &=\int_G |f|(p)\, dm(p)=\|f\|_1,
  \end{align*}
  which shows that $P$ is continuous with respect to the $L_1$-norm.

  Now, since $|P(f)|^2\leq P(|f|^2)$, we have
  \begin{align*}
    \|P(f)\|^2_2&=\int_G |P(f)(p)|^2\, dm(p)\leq \int_G P(|f|^2)(p)\, dm(p)\\
    &=
    \int_G |f(p)|^2\, dm(p)=\|f\|_2^2.
  \end{align*}
  Hence, $P$ is continuous with respect to the $L_2$-norm on
  $\fcc(G)$.

  It is clear that the continuous extension of $P$ to $L_2(G, m)$ is a
  projection. To see that it is orthogonal, let $f_1, f_2 \in \fcc(G)$
  and consider
  \begin{multline*}
    (Pf_1, f_2)_{L_2(G, m)} = \int_G (Pf_1)(p)\overline{f_2(p)}\, dm(p)
    = \int_G P\bigl((Pf_1)(p)\overline{f_2(p)}\bigr)\, d m(p)\\
    = \int_G (Pf_1)(p) \overline{(Pf_2)(p)}\, dm(p)
    = \int_G P(f_1(p) \overline{(Pf_2)(p)})\, dm(p) \\
    = \int_G f_1(p)\overline{(Pf_2)(p)}\, dm(p) = (f_1, Pf_2)_{L_2(G, m)}.
  \end{multline*}
\end{proof}

\begin{remark}
  In what follows, we will identify the space $L_1(Q, \tilde m)$
  (resp., $L_2(Q, \tilde m)$) with a closed subspace of $L_1(G,m)$
  (resp., $L_2(G, m)$).
\end{remark}

\medskip

For a hypergroup $Q$, consider the product hypergroup $Q\times
Q$~\cite{bloom_heyer91:_harmon_analy}, and let $\delta\colon \meas_b
(Q)\rightarrow \meas_b(Q\times Q)$ denote a linear extension of the
map defined by
\begin{equation}
  \label{eq:delta-def-dirac}
  \delta(\varepsilon_s)=\varepsilon_s\otimes \varepsilon_s,
  \quad s\in Q,
\end{equation}
that is, for  $\mu \in \meas_b(Q)$ and $F \in \fcc(Q\times Q)$,
\begin{equation*}
  \pair{\delta(\mu)}{F}=\int_QF(s,s)\, d\mu(s).
\end{equation*}

\begin{definition}
  Let $B$ be a $*$-algebra and $A$ a $C^*$-algebra. A linear map
  $\phi\colon B \rightarrow A$ will be called \emph{positive}, if
  $\phi(b^* b)\geq0$ for all $b \in B$, and \emph{completely
    positive}, if the map $\id \otimes \phi\colon \Mat_n(\Bbb C) \otimes B
  \rightarrow \Mat_n(\Bbb C)\otimes A$ is positive for all $n\in \N$,
  where $\Mat_n(\Bbb C)$ is the $C^*$-algebra of $(n\times n)$-matrices
  over $\Bbb C$.
\end{definition}

It follows from~\cite{takesaki03_i:_theor} that a map $\phi: B\to A$ from
a $*$-algebra $B$ into a $C^\ast$-algebra $A$ is completely
positive if and only if, for any $n\in \N$, $b_i\in B$ and $a_i\in A$,
$i=1,\dots, n$, we have
\begin{equation}
  \label{eq:Takesaki-complete-positivity}
  \sum_{i,j=1}^n a_j^* \phi(b_j^\star b_i) a_i \geq 0.
\end{equation}

\begin{theorem}\label{T:complete-positivity}
  Let $P$ be a conditional expectation on $A_0 = {\mathcal C}_0 (Q)$ that satisfy all the
  conditions in Theorem \ref{Th:hypergroup_from_expectation}. Let $Q$
  be the corresponding hypergroup. Denote by $\lambda_Q$ the left
  regular representation of $\meas_b(Q)$ on $L_2(Q)$,
  \begin{equation*}
    (\lambda_Q(\tilde \mu)\act \tilde f)(t)
    =\int_Q \pair{\tilde f}{\check \varepsilon_{ s}*_Q\varepsilon_t}\,
    d\tilde \mu(s),
    \qquad \tilde \mu \in \meas_b(Q),
    \quad \tilde f \in L_2(Q),
    \quad t\in Q,
  \end{equation*}
  and let $\delta$ be defined by~(\ref{eq:delta-def-dirac}). Then the
  linear map
  \begin{equation*}
    (\lambda_Q\otimes\lambda_Q)\circ \delta\colon L_1(\tilde m)
    \rightarrow B(L_2(Q)\otimes L_2(Q))
  \end{equation*}
 is completely positive.
\end{theorem}

\begin{proof}
  We will identify $L_2(Q)$ with the closed subspace $P(L_2(G))$ of
  $L_2(G)$, and the Banach $*$-algebra $L_1(Q,\tilde m)$ with the
  Banach $*$-subalgebra $P(L_1(G, m))$ of $L_1(G,m)$. With
  such an identification, we have
  \begin{equation}\label{eq:image-of-L(Q)}
    \begin{array}{c}
      (\tilde f, \tilde g)_{L_2(Q)} = (\tilde f, \tilde g)_{L_2(G)},
      \quad \tilde f, \tilde g \in L_2(Q),\\[2mm]
      \lambda_Q(\tilde \mu)\act \tilde f=
      \lambda_G(P^*(\tilde \mu))\act \tilde f,
      \quad \tilde \mu \in \meas_b(Q),\quad \tilde f \in L_2(Q),\\[2mm]
      \delta(P^*(\varepsilon_s))=P^*(\varepsilon_s)\otimes P^*(\varepsilon_s),
      \quad s \in Q.
    \end{array}
  \end{equation}
To prove the theorem, using (\ref{eq:Takesaki-complete-positivity}), it is sufficient to show that
  \begin{align*}
    \sum_{i,j=1}^n&
    \inner{A^*_j\cdot (\lambda_Q \otimes \lambda_Q)\circ
      \delta (\tilde \mu_j^\star *\tilde \mu_i)\cdot A_i\act F}
    {F}_{L_2(Q)\otimes L_2(Q)}\\
    &=
    \sum_{i,j=1}^n
        \inner{ (\lambda_Q \otimes \lambda_Q)\circ
      \delta (\tilde \mu_j^\star *\tilde \mu_i)\cdot A_i\act F}
    {A_j\act F}_{L_2(Q)\otimes L_2(Q)}\geq 0
  \end{align*}
  for any $A_i \in B(L_2(Q)\otimes L_2(Q))$, $\tilde \mu_i \in L_1(\tilde m)$, $F\in
  L_2(Q)\otimes L_2(Q)$, $i=1,\dots, n$.

  Hence letting $A_i F=F_i \in L_2(Q)\otimes L_2(Q)$, and $
  \mu_i=P^*(\tilde \mu_i)\in L_1(m_G)$, and
  using~(\ref{eq:image-of-L(Q)}), we have
  \begin{align*}
    \sum_{i,j=1}^n& \inner{ (\lambda_Q \otimes \lambda_Q)\circ \delta
      (\mu_j^\star *\mu_i)\act F_i}
    { F_j}_{L_2(Q)\otimes L_2(Q)}\\
    &=\sum_{i,j=1}^n \int_G
    \inner{(\lambda_G(P^*(\varepsilon_{\varphi(p)}))\otimes
      \lambda_G (P^*(\varepsilon_{\varphi(p)}))\act F_i}{F_j}_{L_2(G)\otimes L_2(G)}\\
    &\qquad\qquad d({\mu_j}^\star \ast \mu_i )(p).
  \end{align*}

  Let $\mu_i = f_i m$, $f_i\in P(L_1(G))\subset L_1(G)$, $i = 1,\dots
  , n$. Then $\mu_j^\ast = f^\ast_j m$ and
  $$
 {\mu_j}^\star \ast  \mu_i = (f_j^\ast \ast f_i) m.
  $$
  Using left invariance of the measure $m$ we thus have
  \begin{align*}
    \sum_{i,j=1}^n&
    \inner{ (\lambda_Q \otimes \lambda_Q)\circ
      \delta (\mu_j^\star *\mu_i)\act F_i}
    { F_j}_{L_2(Q)\otimes L_2(Q)}\\
    &=\sum_{i,j=1}^n \int_G
    \inner{(\lambda_G(P^*(\varepsilon_{\varphi(p)}))\otimes
      \lambda_G(P^*(\varepsilon_{\varphi(p)})))\act F_i}{F_j}_{L_2(G)\otimes L_2(G)}\\
    &\qquad\qquad \cdot (f_j^\star*f_i)(p)\, dm(p)\\
    &=\sum_{i,j=1}^n \int_{G^2}
    \inner{(\lambda_G(P^*(\varepsilon_{\varphi(p)}))\otimes
      \lambda_G(P^*(\varepsilon_{\varphi(p)})))\act F_i}{F_j}_{L_2(G)\otimes L_2(G)}\\
    &\qquad\qquad \cdot f^*_j(q) f_i(q^{-1}p)\, dm(q)dm(p)\\
    &=\sum_{i,j=1}^n \int_{G^2}
    \inner{(\lambda_G(p)\otimes
      \lambda_G(p))\act F_i}{F_j}_{L_2(G)\otimes L_2(G)}\\
    &\qquad\qquad \cdot f^*_j(q) f_i(q^{-1}p)\, dm(q)dm(p)\\
    &=\sum_{i,j=1}^n \int_{G^2}
    \inner{(\lambda_G(qp)\otimes
      \lambda_G(qp))\act F_i}{F_j}_{L_2(G)\otimes L_2(G)}\\
    &\qquad\qquad \cdot f^*_j(q) f_i(p)\, dm_G(q)dm_G(p)\\
    &=\sum_{i,j=1}^n \int_{G^2}
    \inner{(\lambda_G(p)\otimes
      \lambda_G(p))\act F_i}
    {(\lambda_G(q^{-1})
      \otimes \lambda_G(q^{-1})\act F_j}_{L_2(G)\otimes L_2(G)}\\
    &\qquad\qquad \cdot f^*_j(q) f_i(p)\, dm_G(q)dm_G(p)\\
    &=\sum_{i,j=1}^n \int_{G^2}
    \inner{(\lambda_G(p)\otimes
      \lambda_G(p))\act F_i}
    {(\lambda_G(q^{-1})
      \otimes \lambda_G(q^{-1})\act F_j}_{L_2(G)\otimes L_2(G)}\\
    &\qquad\qquad \cdot \overline f_j(q^{-1})
    f_i(p)\, dm_G(q^{-1})dm_G(p)\\
    &=
    \Bigl\|\sum_{i=1}^n \int_G f_i(p) (\lambda_G (p)\otimes
    \lambda_G (p))\, dm_G(p) \act F_i\Bigr\|^2_{L_2(G)\otimes L_2(G)}
    \geq 0.
  \end{align*}
\end{proof}

\begin{corollary}
  The Fourier space $\Fs(Q)$ is a Banach algebra.
\end{corollary}

\begin{proof}
The statement of this corollary follows from the results of \cite{ChKP2015}.
\end{proof}



\end{document}